\documentclass[a4paper]{article}
\usepackage{fullpage}
\usepackage{tikz}

\usepackage[pdftex,breaklinks,colorlinks,
linkcolor=blue,
citecolor=blue,
urlcolor=blue]{hyperref}
\newcommand{\pr}[1]{{\color{black}#1}}

\usepackage{enumitem}
\usepackage{amsmath,amssymb,amsfonts,amsthm}

\newtheorem{theorem}{Theorem}
\newtheorem{conjecture}[theorem]{Conjecture}
\newtheorem{lemma}[theorem]{Lemma}
\newtheorem{proposition}[theorem]{Proposition}
\newtheorem{corollary}[theorem]{Corollary}

\theoremstyle{definition}
\newtheorem{definition}[theorem]{Definition}
\newtheorem{remark}[theorem]{Remark}

\author{Pablo Romero\footnote{Facultad de Ingenier\'ia, Universidad de la Rep\'ublica, Montevideo, Uruguay. E-mail address: \texttt{promero@fing.edu.uy}.}\footnote{Facultad de Ciencias Exactas y Naturales, Universidad de Buenos Aires. E-mail address: \texttt{promero@ic.fcen.uba.ar}.}\qquad Louis Petingi\footnote{College of Staten Island, City University of New York (CUNY). E-mail address: \texttt{louis.petingi@csi.cuny.edu}.}}

\date{}

\makeatletter%
\begin{document}

\title{Trace-minimal graphs maximizing the number of spanning trees\thanks{A preliminary version of this work appeared in the proceedings of the XIII Latin-American Algorithms, Graphs and Optimization Symposium~\cite{RomeroLAGOS1-2025}.}}

\maketitle

\begin{abstract}\let\thefootnote\relax
A classical problem in spectral graph theory is to determine, among all simple graphs with a fixed number of vertices and edges, those that maximize the number of spanning trees. Such graphs are called $t$-optimal. Petingi and Rodríguez [Discr. Math. 244 (2002), 351--373]  reduced the construction of infinitely many $t$-optimal graphs to the determination of almost-regular graphs with a minimum number of induced $3$-paths, known as almost-regular $\nu$-min graphs. They conjectured that the class of $d$-regular $\nu$-min graphs is invariant under the addition or deletion of copies of the complete graph $K_{d+1}$, which would provide a method to construct large families of $t$-optimal graphs. This conjecture was recently verified for $d\in \{1,2,3,4,5\}$ by Gong \textit{et~al.} [Discr. Math. 349(4), 2026].

In this work, we establish an asymptotic version of this conjecture. We prove that the class of $d$-regular $\nu$-min graphs on at least $n$ vertices is invariant under the addition or deletion of copies of $K_{d+1}$ whenever $d\geq 3$ and $n\geq d^3$. In addition, we introduce a new method to construct $t$-optimal graphs based on the concept of trace-minimal graphs, originally introduced by Ábrego \textit{et~al.} [Linear Algebra Appl. 412 (2006), 161--221]. As a consequence, we obtain infinitely many graphs that are both trace-minimal and $t$-optimal.
\end{abstract}

\renewcommand{\labelitemi}{--}

\section{Introduction} \label{intro}

In the first half of the nineteenth century, Kirchhoff, in his study of linear resistive circuits, proved that the number of spanning trees of a graph equals any cofactor of its Laplacian matrix~\cite{Biggs-1993}. As a consequence, the number of spanning trees of a graph, sometimes called the \emph{tree-number}, can be computed efficiently. This raises a natural extremal question: given integers $n$ and $m$, how can one construct simple graphs on $n$ vertices and $m$ edges that maximize the tree-number? Graphs attaining this maximum are called \emph{$t$-optimal graphs}, and they play a role in network reliability analysis~\cite{Boesch-1986,RomeroSafe-2026,Romero-2026}.

Several classes of $t$-optimal graphs are known. Kelmans and Chelnokov~\cite{Kelmans-1974} proved that the complete graph $K_n$ minus a matching is $t$-optimal, and Cheng~\cite{Cheng-1981} showed that every regular complete multipartite graph is $t$-optimal. Boesch~\cite{Boesch-1986} conjectured that every $t$-optimal graph is almost regular, a statement later confirmed asymptotically by Petingi and Rodríguez~\cite{Petingi-2002}. They generalized Cheng’s result by proving that every almost-regular complete multipartite graph is $t$-optimal. 
Additionally, they provided a method 
to construct $t$-optimal graphs 
by means of the determination of 
almost-regular graphs that minimize the number of induced $3$-paths, known as almost-regular $\nu$-min graphs. Petingi \textit{et~al.}~\cite{Petingi-1998} and later Petingi and Rodríguez~\cite{Petingi-2002} exhibited infinitely many $t$-optimal graphs on $n$ vertices and $m$ edges for the ranges $\binom{n}{2}-n \leq m \leq \binom{n}{2}-n/2$ and $\binom{n}{2}-3n/2 \leq m \leq \binom{n}{2}-n$, respectively.

A central open question arising from~\cite{Petingi-2002} concerns the structure of $d$-regular $\nu$-min graphs. Petingi and Rodríguez conjectured that the class of all $d$-regular $\nu$-min graphs is invariant under the addition or deletion of copies of the complete graph $K_{d+1}$. If true, this would yield a recursive mechanism to construct infinitely many regular $\nu$-min graphs, and consequently, large families of $t$-optimal graphs. This conjecture has recently been verified for $d\in \{1,2,3,4,5\}$ by Gong \textit{et~al.}~\cite{Gong-2026}, who also obtained new infinite families of $t$-optimal graphs whose complements are either $4$-regular or $5$-regular. Additional constructions have been obtained for graph classes of small corank. In this direction, Kahl~\cite{Kahl-2022} introduced the concept of Tutte-maximum graphs, which are also $t$-optimal, and Kahl and Luttrell~\cite{Kahl-2023} determined all such graphs with corank at most $3$.

A key ingredient for our approach comes from a different line of research. Ábrego \textit{et~al.}~\cite{Abrego-2006}, in the context of $D$-optimal weighing designs, introduced the notion of \emph{trace-minimal graphs}. These graphs are defined through a lexicographic minimization of trace sequences of powers of the adjacency matrix, and were used to construct $D$-optimal weighing designs (i.e., how to estimate the weight of some objects with minimum uncertainty). As we show in this work, this concept is also well suited to the study of extremal problems for the number of spanning trees.\\

The purpose of this article is twofold. First, we establish an asymptotic version of the conjecture of Petingi and Rodríguez. Specifically, we prove that the class of $d$-regular $\nu$-min graphs on at least $n$ vertices is invariant under the addition or deletion of copies of $K_{d+1}$ whenever $d\geq 3$ and $n\geq d^3$ (see Proposition~\ref{proposition:conj}). 

Second, we introduce a new method to construct $t$-optimal graphs based on trace-minimal graphs. By relating trace-minimality to a Laplacian-based lexicographic minimization, we obtain a framework that refines the method of Petingi and Rodríguez~\cite{Petingi-2002}. As a consequence, we prove the existence of infinitely many graphs that are simultaneously trace-minimal and $t$-optimal (see Theorem~\ref{theorem:summary}), yielding infinitely many extremal graphs.\\

The article is organized as follows. Section~\ref{section:concepts} presents basic concepts in graph theory. Section~\ref{section:background} reviews results from spectral graph theory and summarizes key ideas from~\cite{Petingi-2002,Abrego-2006}, together with auxiliary results on trace-minimal graphs. The main results are given in Section~\ref{section:main}, where we prove the asymptotic version of the conjecture (Proposition~\ref{proposition:conj}) and develop the new construction method based on trace-minimality (Theorem~\ref{theorem:summary}). Finally, in Section~\ref{section:toptimal}, we construct infinitely many $6$-regular and $7$-regular graphs whose complements are both trace-minimal and $t$-optimal.

\section{Concepts}\label{section:concepts}

Let $\mathcal{S}_{n,m}$ be the class of simple graphs on $n$ vertices and $m$ edges, and let $G\in \mathcal{S}_{n,m}$. We denote its vertex and edge sets by $V(G)$ and $E(G)$, respectively, and write $V(G)=\{v_1,v_2,\ldots,v_n\}$. Two vertices $v_i$ and $v_j$ are said to be \emph{adjacent} if $v_iv_j \in E(G)$; in this case, the edge $v_iv_j$ is \emph{incident} to both $v_i$ and $v_j$. The \emph{degree} of a vertex $v_i$, denoted by $d_i$, is the number of edges incident to $v_i$. The \emph{degree sequence} of $G$ is $d(G)=(d_1,d_2,\ldots,d_n)$.

We say that $G$ is \emph{$d$-regular} if $d_i=d$ for all $i\in \{1,2,\ldots,n\}$, and we denote by $\mathcal{R}_d(n)$ the class of $d$-regular graphs on $n$ vertices. We say that $G$ is \emph{almost-regular} if the degrees of any two of its vertices differ by at most $1$, and we denote by $\mathcal{A}_{n,m}$ the class of almost-regular graphs in $\mathcal{S}_{n,m}$.

We denote by $I_n$, $0_n$, and $J_n$ the identity matrix, the all-zero matrix, and the all-one matrix of size $n \times n$, respectively. The \emph{trace} of a matrix $M$, denoted $\mathrm{tr}(M)$, is the sum of its diagonal entries. The \emph{adjacency matrix} of $G$ is $A(G)=(a_{i,j})_{1\leq i,j\leq n}$, where $a_{i,j}=1$ if $v_iv_j \in E(G)$ and $a_{i,j}=0$ otherwise. Let $D(G)$ be the diagonal matrix with diagonal entries $d_1,\ldots,d_n$. The \emph{Laplacian matrix} of $G$ is $L(G)=D(G)-A(G)$, and its characteristic polynomial is $P_G(x)=\det(xI_n-L(G))$. The \emph{Laplacian spectrum} of $G$ is the multiset consisting of all roots of $P_G(x)$.

For each $S \subseteq V(G)$, the \emph{subgraph of $G$ induced by $S$}, denoted $G[S]$, is the graph with vertex set $S$ and all edges of $G$ with both endpoints in $S$. Let $t(G)$, $\tau(G)$, and $\nu(G)$ denote the number of spanning trees, induced triangles, and induced $3$-paths of $G$, respectively. We say that $G$ is \emph{$t$-optimal} if $t(G)\geq t(H)$ for all $H\in \mathcal{S}_{n,m}$.

If $G\in \mathcal{A}_{n,m}$, we say that $G$ is \emph{$\nu$-min} if $\nu(G)\leq \nu(H)$ for all $H\in \mathcal{A}_{n,m}$, and \emph{$\tau$-max} if $\tau(G)\geq \tau(H)$ for all $H\in \mathcal{A}_{n,m}$. The \emph{complement} of $G$ is denoted by $\overline{G}$. The union of two disjoint graphs $G$ and $H$ is denoted by $G \cup H$. The \emph{join} of $G$ and $H$, denoted $G \wedge H$, is obtained from $G \cup H$ by adding all edges between $V(G)$ and $V(H)$. We denote by $G^{(n)}$ the join of $n$ disjoint copies of $G$, and by $n(G)$ the number of vertices of $G$.

We denote by $\mathcal{S}(G)$, $\mathcal{A}(G)$, and $\mathcal{R}(G)$ the classes of simple, almost-regular, and regular graphs with the same number of vertices and edges as $G$, respectively. We denote by $P_n$, $C_n$, $K_n$, and $E_n$ the path, cycle, complete graph, and edgeless graph on $n$ vertices, respectively. The complete bipartite graph $K_{r,s}$ is defined as $E_r \wedge E_s$.

\section{Related work}\label{section:background}
In this section we present some elementary results concerning spectral graph theory that will be used throughout this work. Then, we revisit the method developed by Petingi and Rodríguez~\cite{Petingi-2002} to construct $t$-optimal graphs.  
Finally, we include the concept of trace-minimal graphs 
as defined by \'Abrego \textit{et~al.}~\cite{Abrego-2006}, and present 
some results on the existence of trace-minimal graphs which will be useful for our purposes.\\

Lemma~\ref{lemma:basics} can be proved using the definition of the Laplacian matrix of a graph and linear algebra.
\begin{lemma}[Anderson~\cite{Anderson-1985}]\label{lemma:basics}
Let $G$ be any graph in $\mathcal{S}_{n,m}$. Let $\Delta(G)$ be the maximum degree of all vertices in $G$. All the following assertions hold.
\begin{enumerate}
\item\label{prop1} The Laplacian matrix $L(G)$ is both symmetric and positive semidefinite. 
\item\label{prop2} If $\lambda_1,\ldots,\lambda_n$ are the $n$ eigenvalues of $L(G)$ such that  $0=\lambda_1 \leq \lambda_2 \leq \cdots \leq \lambda_n$, then $P_{\overline{G}}(x)=x\prod_{i=2}^{n}(x-(n-\lambda_i))$.
\item\label{prop3} Each eigenvalue of $L(G)$ lies in the interval $[0,n]$.
\item\label{prop4} Each eigenvalue of $L(G)$ lies in the interval $[0,2\Delta(G)]$.
\item\label{prop5} The multiplicity of the eigenvalue $0$ of $L(G)$ equals the number of components of $G$.
\item\label{prop6} If $\overline{G}$ is connected then $n$ is not an eigenvalue of $L(G)$.
\end{enumerate}
\end{lemma}

Recall that $E_n$ is the edgeless graph on $n$ vertices. 
By Lemma~\ref{lemma:basics}\ref{prop5}, the multiplicity of the eigenvalue $0$ of $L(E_n)$ equals $n$.   
Applying Lemma~\ref{lemma:basics}\ref{prop2} to 
$E_n$ gives Remark~\ref{remark:complete}.
\begin{remark}\label{remark:complete}
The characteristic polynomial of $L(K_n)$ is 
$P_{K_n}(x)=x(x-n)^{n-1}$.
\end{remark}

Now, let $G$ be any graph in $\mathcal{S}_{n,m}$ such that  
$P_G(x)=x \prod_{i=2}^{n}(x-\lambda_i)$. 
For each $i,j\in \{1,\ldots,n\}$, let $c_{ij}(L(G))$ be the cofactor of $L(G)$ at entry $(i,j)$. By Kirchhoff’s theorem, $t(G)=c_{ij}(L(G))$ for all $i,j\in \{1,\ldots,n\}$. Moreover, the coefficient of $x$ in $P_G(x)$ equals $(-1)^{n-1}\sum_{i=1}^n c_{ii}(L(G))$. Since this coefficient is also $(-1)^{n-1}\prod_{i=2}^n \lambda_i$, we obtain that
\begin{equation*}
t(G)=c_{11}(L(G))=\frac{1}{n}\sum_{i=1}^{n}c_{ii}(L(G))= \frac{1}{n}\prod_{i=2}^{n}\lambda_i = \frac{1}{n^2}(n\prod_{i=2}^{n}\lambda_i) =\frac{1}{n^2}P_{\overline{G}}(n),    
\end{equation*}
where Lemma~\ref{lemma:basics}\ref{prop2} was used in the last step. Therefore, the number of spanning trees of a graph $G$ is determined by its Laplacian spectrum and Lemma~\ref{lemma:t} holds.
\begin{lemma}\label{lemma:t}
If $G$ is a simple graph on $n$ vertices then $t(G)=\frac{1}{n^2}P_{\overline{G}}(n)$.    
\end{lemma}

Let $G$ be any graph in $\mathcal{S}_{n,m}$. Suppose that $\overline{G}$ is connected. In the following, we derive a closed-form expression for the tree-number of $\overline{G}$, i.e.,  $t(\overline{G})$. Let  $P_G(x)=\prod_{i=1}^{n}(x-\lambda_i)$, where $\lambda_1,\lambda_2,\ldots,\lambda_n$ are the eigenvalues of $L(G)$. As $\overline{G}$ is connected,  Lemma~\ref{lemma:basics}\ref{prop6} gives that $n$ is not an eigenvalue of $L(G)$. 
Lemma~\ref{lemma:basics}\ref{prop3} gives that each of the eigenvalues of $L(G)$ lies in the interval $[0,n)$. As a consequence, the function $P_G(x)/x^n$ is positive whenever $x\geq n$. 
Additionally, when $x\geq n$ it follows that $x$ is greater than each of the eigenvalues of $L(G)$ thus 
\begin{align}\label{eq:nodrop}
-\log\left(\frac{P_G(x)}{x^n}\right) &= -\log\left(\prod_{i=1}^{n}(1-\frac{\lambda_i}{x})\right) = \sum_{i=1}^{n}-\log\left(1-\frac{\lambda_i}{x}\right) \nonumber\\
&=\sum_{i=1}^{n}\sum_{k=1}^{\infty}\frac{\lambda_i^k}{kx^k} = \sum_{k=1}^{\infty}\frac{\sum_{i=1}^n \lambda_i^k}{kx^k} = \sum_{k=1}^{\infty}\frac{\mathrm{tr}(L(G)^k)}{kx^k} = \sum_{k=1}^{\infty}\frac{\ell_k(G)}{kx^k},
\end{align}
where $\ell_k(G)=\mathrm{tr}(L(G)^k)$ is the \emph{Laplacian sequence} of $G$, and the power series expansion $-\log(1-\frac{\lambda_i}{x})=\sum_{k=1}^{\infty}\frac{\lambda_i^k}{kx^k}$ holds since $x \geq n > |\lambda_i|$. 
Petingi and Rodríguez found lower bounds for $\ell_k(G)$. 
\begin{lemma}[Petingi and Rodríguez~\cite{Petingi-2002}]\label{lemma:gaps}
If $G$ is a simple graph with degree sequence $(d_1,d_2,\ldots,d_n)$ and $k$ is any positive integer then  
$\ell_k(G) \geq \sum_{i=1}^{n}d_i(d_i+1)^{k-1}$. The equality holds for each positive integer $k$ if and only if $G$ is a disjoint union of complete graphs.
\end{lemma}
Lemma~\ref{lemma:gaps} motivates to define the \emph{gap sequence} $g_k(G)$ of a graph $G$ with degree sequence $(d_1,d_2,\ldots,d_n)$ 
as $g_k(G)=\ell_k(G)-\sum_{i=1}^{n}d_i(d_i+1)^{k-1}$. Petingi and Rodríguez observed that $g_1(G)=g_2(G)=0$ and $g_3(G)=2\nu(G)$. Consequently, by Equation~\ref{eq:nodrop} it follows that for all $x\geq n$,
\begin{align}\label{ineq:chain}
-\log\left(\frac{P_G(x)}{x^n}\right) &= \sum_{k=1}^{\infty}\frac{\ell_k(G)}{kx^k} \geq 
\sum_{k=1}^{\infty}\frac{\sum_{i=1}^{n}d_i(d_i+1)^{k-1}}{kx^k} + \frac{2\nu(G)}{3x^3} \nonumber\\
&=\sum_{i=1}^{n}\frac{d_i}{d_i+1}\sum_{k=1}^{\infty}\frac{(d_i+1)^k}{kx^k} + \frac{2\nu(G)}{3x^3} = -\sum_{i=1}^{n}\frac{d_i}{d_i+1}\log\left(1- \frac{d_i+1}{x}\right)+ \frac{2\nu(G)}{3x^3} \nonumber\\
&= -\log\left(e^{-\frac{2\nu(G)}{3x^3}}\prod_{i=1}^{n}\left(1-\frac{d_i+1}{x} \right)^{\frac{d_i}{d_i+1}} \right). 
\end{align}
Rearranging terms, we obtain that
\begin{equation}\label{eq:poly}
P_G(x) \leq x^n e^{-\frac{2\nu(G)}{3x^3}}\prod_{i=1}^{n}\left( 1- \frac{d_i+1}{x}\right)^{\frac{d_i}{d_i+1}}.    
\end{equation}
The authors obtained the following result replacing~\eqref{eq:poly} into the expression for $t(\overline{G})$ given in Lemma~\ref{lemma:t}. 
\begin{lemma}[Petingi and Rodríguez~\cite{Petingi-2002}]\label{lemma:PetingiRodriguez}
If $G$ is a simple graph on $n$ vertices with degree sequence $(d_1,d_2,\ldots,d_n)$ then  
\begin{equation}\label{eq:bound1}
t(\overline{G}) \leq n^{n-2}e^{-\frac{2\nu(G)}{3n^3}}\prod_{i=1}^{n}\left( 1- \frac{d_i+1}{n}\right)^{\frac{d_i}{d_i+1}}.
\end{equation}
The equality occurs if and only if $G$ is a union of complete graphs.
\end{lemma}
Let $f(d,x)$ be the product appearing on the right-hand side of equation~\eqref{eq:poly},
\begin{equation}\label{eq:f}
f(d,x) = \prod_{i=1}^{n}\left( 1- \frac{d_i+1}{x}\right)^{\frac{d_i}{d_i+1}}. 
\end{equation}
The authors proved that $f(d,x)$ is maximized among simple graphs $G$ in $\mathcal{S}_{n,m}$ precisely when $G$ is almost-regular. 
\begin{lemma}[Petingi and Rodríguez~\cite{Petingi-2002}]\label{lemma:f}
If $H \in \mathcal{S}_{n,m}-\mathcal{A}_{n,m}$ and $G \in \mathcal{A}_{n,m}$, then $f(d(H),x)< f(d(G),x)$ for all $x\geq n$.
\end{lemma}
A longstanding conjecture proposed by Boesch~\cite{Boesch-1986} states that every $t$-optimal graph is almost-regular. Petingi and Rodríguez proved that the Boesch conjecture holds in an asymptotic sense which is precisely stated in Proposition~\ref{proposition:PetingiRodriguez-almostregular}. The following notation will be used throughout this article. For each almost-regular graph $G_0$ whose vertices have degree $d-1$ or $d$ and each pair of nonnegative integers $p$ and $q$, we let 
$G_0(p,q)$ be the graph $G_0 \cup pK_{d+1} \cup qK_d$. Note that  $G_0(p,q)$ and $\overline{G_0(p,q)}$ are almost-regular graphs.
\begin{proposition}[Petingi and Rodríguez~\cite{Petingi-2002}]\label{proposition:PetingiRodriguez-almostregular}
Let $G_0$ be an almost-regular graph having    
$r$ and $s$ vertices with degrees $d-1$ and $d$, respectively. Then there exists a positive integer $n_0$ such that each graph $H$ in $\mathcal{S}(G_0(p,q))$ with at least $n_0$ vertices that is not almost-regular satisfies that $t(\overline{G_0(p,q)})>t(\overline{H})$.
\end{proposition}

During the proof of Proposition~\ref{proposition:PetingiRodriguez-almostregular} the authors obtained that $n_0$ can be chosen as any integer satisfying simultaneously the three following conditions: (i) $n_0>d$, (ii) $v_1(x)<0$ for all $x\geq n_0$, and (iii) $v_2(x)<0$ for all $x\geq n_0$, where 
\begin{align*}
v_1(x) &= -2+\frac{1}{x}\sum_{j=0}^{\infty}\frac{b_{j+2}}{x^j},\\
v_2(x) &= -2+\frac{1}{x}\sum_{j=0}^{\infty}\frac{c_{j+2}}{x^j},
\end{align*}
with $b_j=\mathrm{tr}(L(G_0)^{j+1})-(d-2)(d-1)^j-(d-1)d^j(r-2)-d(d+1)^j(s+1)$ and $c_j=\mathrm{tr}(L(G_0)^{j+1})-(d-1)d^j(r+1)-d(d+1)^j(s-2)-(d+1)(d+2)^j$.

\begin{remark}\label{remark:cobertura}
A possible choice for $n_0$ is $2d+n(G_0)(2d)^3$. 
In fact, by Lemma~\ref{lemma:basics}\ref{prop4}, 
each eigenvalue of $L(G_0)$ lies in the interval $[0,2d]$ thus  
$\mathrm{tr}(L(G_0)^{j+1})\leq n(G_0)(2d)^{j+1}$. Then 
$b_j \leq n(G_0)(2d)^{j+1}$ and $c_j\leq n(G_0)(2d)^{j+1}$ and 
both functions $v_1(x)$ and $v_2(x)$ are upper bounded by 
$-2+\frac{1}{x}\sum_{j=0}^{\infty}\frac{n(G_0)(2d)^{j+3}}{x^j} = -2 + \frac{n(G_0)(2d)^3}{x-2d}$, which is negative when $x\geq 2d+n(G_0)(2d)^3$.
\end{remark}

Proposition~\ref{proposition:PetingiRodriguez} is a refinement of Proposition~\ref{proposition:PetingiRodriguez-almostregular}. It takes into account Lemma~\ref{lemma:PetingiRodriguez} which gives priority to $\nu$-min graphs among all graphs in $\mathcal{A}_{n,m}$. 
\begin{proposition}[Petingi and Rodríguez~\cite{Petingi-2002}]\label{proposition:PetingiRodriguez}
Let $G_0$ be an almost-regular graph all of whose vertices have degree $d-1$ or $d$. Suppose that $G_0(p,q)$ is $\nu$-min for all $p$ and $q$. Then there exists a positive integer $n_0$ such that each graph $H$ in $\mathcal{S}(G_0(p,q))$ with at least $n_0$ vertices that is not $\nu$-min satisfies that $t(\overline{G_0(p,q)})>t(\overline{H})$.
\end{proposition}

The methodology proposed by Petingi and Rodríguez to find $t$-optimal graphs 
consists of finding an almost-regular graph $G_0$ satisfying the conditions of Proposition~\ref{proposition:PetingiRodriguez}, and then identifying, among all $\nu$-min graphs in $\mathcal{A}(G_0(p,q))$, those graphs whose complements have the maximum number of spanning trees. 

The authors concluded the article with the following conjecture, which, if true, would yield all regular $\nu$-min graphs in the infinitely many classes $\mathcal{R}(G_0(p,0))$ by determining $\nu$-min graphs only in $\mathcal{R}(G_0)$. For brevity we denote  $G_0(p)$ for $G_0\cup pK_{d+1}$.

\begin{conjecture}[Petingi and Rodríguez~\cite{Petingi-2002}]\label{conjecture:PR}
Let $G_0$ be in $\mathcal{R}_d(n)$. If $G_0$ is $\nu$-min in $\mathcal{R}_d(n)$ then, for each positive integer $p$, the graph $G_0(p)$ is $\nu$-min in $\mathcal{R}(G_0(p))$. Conversely, if some graph $G'$ is $\nu$-min in $\mathcal{R}(G_0(p))$ then there exists $G_0'$ in $\mathcal{R}_d(n)$ such that $G'=G_0'(p)$.
\end{conjecture}

Our study of Conjecture~\ref{conjecture:PR} is based on  Remark~\ref{remark:counting}.
\begin{remark}\label{remark:counting}
The authors in~\cite{Petingi-2002} observed that for each $G$ in $\mathcal{S}_{n,m}$ with degree sequence $(d_1,\ldots,d_n)$ it holds that $\sum_{i=1}^{n}\binom{d_i}{2}=3\tau(G)+\nu(G)$. In fact, each 
induced $3$-path includes precisely one pair of adjacent edges, while each triangle includes precisely $3$ 
pairs of adjacent edges. Therefore, a graph is  $\nu$-min in $\mathcal{A}_{n,m}$ if and only if it is $\tau$-max in $\mathcal{A}_{n,m}$.
\end{remark}
We now review the methodology employed by Petingi and Rodríguez to find $t$-optimal graphs~\cite{Petingi-2002}. Our point of departure is Remark~\ref{remark:key}.

\begin{remark}\label{remark:key}  
In the inequality of expression~\eqref{ineq:chain}, each of the terms $\ell_k(G)$ such that $k\neq 3$ 
is replaced by its corresponding lower bound $\sum_{i=1}^{n}d_i(d_i+1)^{k-1}$. 
However, by Lemma~\ref{lemma:gaps}, $\ell_k(G)=\sum_{i=1}^{n}d_i(d_i+1)^{k-1}+g_k(G)$, and sharper bounds for $t(\overline{G})$ could be obtained.
\end{remark}

We will apply Remark~\ref{remark:key} to find  sharper bounds 
for $t(\overline{G})$ than the bound given 
in Lemma~\ref{lemma:PetingiRodriguez}. As a consequence, a generalization of Proposition~\ref{proposition:PetingiRodriguez} will be given (see Proposition~\ref{proposition:generalize}). A key concept to achieve this goal is that of a trace-minimal graph~\cite{Abrego-2006}. A lexicographic order among sequences of real numbers is first required.
\begin{definition}
For each pair of sequences of real numbers $(b_i)_{i\in \mathbb{Z}^+}$ and $(c_i)_{i\in \mathbb{Z}^+}$  
we write  $(b_i) \preceq (c_i)$ when precisely one of the two following conditions holds:
\begin{enumerate}[label=(\roman*)]
\item For each positive integer $k$ it holds that $b_k = c_k$ (i.e., the sequences are identical),  
\item There exists a positive integer $j$ such that for each $k\in \{1,2,\ldots,j-1\}$ it holds that $b_k= c_k$ but $b_j<c_j$.
\end{enumerate}
\end{definition}

For each $G$ in $\mathcal{S}_{n,m}$ we define its \emph{adjacency sequence} $(a_i(G))_{i\in \mathbb{Z}^+}$ as $a_i(G)=\mathrm{tr}(A(G)^i)$.  

\begin{definition}[\'Abrego \textit{et~al.}~\cite{Abrego-2006}]
A graph $G$ in $\mathcal{R}_{d}(n)$ is \emph{trace-minimal} if for each $H$ in $\mathcal{R}_{d}(n)$ it holds that $(a_i(G)) \preceq (a_i(H))$. 
\end{definition}
To close this section, we will give a list of results concerning trace-minimal graphs. First, we observe that trace-minimal graphs have the least number of triangles.
\begin{remark}\label{remark:triangles}
For each $G$ in $\mathcal{S}_{n,m}$ it is known that $a_1(G)=0$, $a_2(G)=m$ and $a_3(G)=6\tau(G)$. 
Consequently, each trace-minimal $G$ 
in $\mathcal{R}_d(n)$ has the least number of triangles among all graphs in $\mathcal{R}_d(n)$. 
\end{remark}

For each vertex $v$ in $V(G)$, let $\tau(G,v)$ be the number of triangles in $G$ that contain $v$.



\begin{lemma}[\'Abrego \textit{et~al.}~\cite{Abrego-2006}]\label{lemma:component}
If $G$ is any $\tau$-max graph in $\mathcal{R}_d(n)$ then either $K_{d+1}$ is a component of $G$ or $\tau(G,v)\leq \binom{d}{2}-2$ for each vertex $v$ in $G$. 
\end{lemma}

Denote $\tau_{d}(n)=\min\{\tau(G): G\in \mathcal{R}_d(n)\}$. Following the notation in~\cite{Abrego-2006,Saegusa-2007}, 
let $S(n,d)$ be any trace-minimal graph in $\mathcal{R}_d(n)$. 

\begin{lemma}[\'Abrego \textit{et~al.}~\cite{Abrego-2006}]\label{lemma:test}
Let $G$ be in $\mathcal{R}_{n-d-1}(n)$ for some integer $d$ such that $d\geq 3$. 
Let $q$ and $r$ be the only integers such that $r\in \{0,1,\ldots,d\}$ and $n=q(d+1)+r$. 
If $G$ is trace-minimal 
then $G$ is isomorphic to  $E_{d+1}\wedge S(n-d-1,n-2d-2)$ 
or the inequality $n\leq \mu(d,r)$ must be satisfied, where the function $\mu(d,r)$ is defined as follows,
\begin{equation}\label{eq:w}
\mu(d,r) = \frac{r}{4}((d+1)^2-r^2) + \frac{3}{2}\tau_{r}(d+1+r).    
\end{equation}
\end{lemma}

The inequality $n\leq \mu(d,r)$ from Lemma~\ref{lemma:test} does not hold when 
$n$ is large enough. The following result is a corollary of Lemma~\ref{lemma:test}.
\begin{lemma}[\'Abrego \textit{et~al.}~\cite{Abrego-2006}]\label{lemma:corollary}
Let $d$ and $n$ be integers such that $d\geq 3$ and $n\geq d^3$. 
If $G$ is trace-minimal in $\mathcal{R}_{n-d-1}(n)$ then  $G=E_{d+1}\wedge S(n-d-1,n-2d-2)$.
\end{lemma}

Lemma~\ref{lemma:trianglefree} gives a sufficient condition 
under which the last term in the right-hand side of Equation~\eqref{eq:w} vanishes. A proof of Lemma~\ref{lemma:trianglefree} follows from Mantel’s theorem which states that each triangle-free graph on $n$ vertices has at most $\lfloor n^2/4 \rfloor$ edges, and the equality is attained only by $K_{\lfloor n/2 \rfloor, \lceil n/2 \rceil}$. If $n$ is any positive even integer and $2d\leq n$, then we can remove perfect matchings to  $K_{n/2,n/2}$ and obtain a triangle-free graph in $\mathcal{R}_d(n)$ for any choice of integer $d$ such that $2d\leq n$. 
\begin{lemma}\label{lemma:trianglefree}
Let $d$ and $n$ be positive integers such that $n$ is even 
and $d\leq n-1$. Then there exists a triangle-free graph in $\mathcal{R}_d(n)$ if and only if $d\leq n/2$.
\end{lemma}

\'Abrego \textit{et~al.}~\cite{Abrego-2006} proved that 
strongly regular graphs without triangles, some cages, and some bipartite graphs are trace-minimal graphs. We will use Lemma~\ref{lemma:bipartitem} and  Lemma~\ref{lemma:cases}.
\begin{lemma}[\'Abrego \textit{et~al.}~\cite{Abrego-2006}]\label{lemma:bipartitem}
The only trace-minimal graph in $\mathcal{R}_{n-1}(2n)$ is $K_{n,n}-nK_2$  (i.e., the complete bipartite graph $K_{n,n}$ minus a perfect matching).    
\end{lemma}

Figure~\ref{fig:sporadic} depicts the graphs $S(11,4)$, $S(12,4)$ and $S(13,6)$. Both graphs $S(11,4)$ and $S(12,4)$ are triangle-free graphs, while $S(13,6)$ has precisely $6$ triangles. Figure~\ref{fig:sporadic} shows in bold lines each of the edges in $S(13,6)$ which are included in a triangle. 
\begin{lemma}[\'Abrego \textit{et~al.}~\cite{Abrego-2006}]\label{lemma:cases}
The only trace-minimal graphs in $\mathcal{R}_{4}(11)$, $\mathcal{R}_{4}(12)$ and $\mathcal{R}_{6}(13)$ are $S(11,4)$, $S(12,4)$ and $S(13,6)$, respectively.     
\end{lemma}
Saegusa~\cite{Saegusa-2007} achieved in finding additional trace-minimal graphs. 
\begin{proposition}[Saegusa~\cite{Saegusa-2007}]\label{proposition:saegusa}
For each integer $n$ such that $n\geq 2$ it holds that the only trace-minimal graph in $\mathcal{R}_{n+1}(2n)$ is $E_{n-1} \wedge C_{n+1}$.    
\end{proposition}
\noindent Conjecture~\ref{conjecture:Saegusa} would yield a way to construct trace-minimal graphs in $\mathcal{R}_{n+k}(2n)$ when $k\in \{2,\ldots,n-1\}$.
\begin{conjecture}[Saegusa~\cite{Saegusa-2007}]\label{conjecture:Saegusa}
In each nonempty set $\mathcal{R}_{n+k}(2n)$ where $k\in \{2,\ldots,n-1\}$, 
the graph $E_{n-k}\wedge S(n+k,2k)$ is trace-minimal.
\end{conjecture}

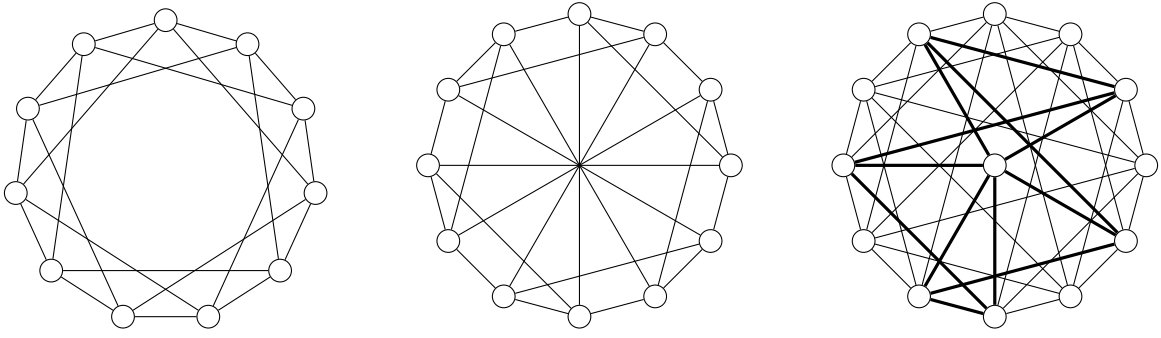
\begin{figure}
\begin{center}
\scalebox{1}{
%
\begin{tikzpicture}[
  scale=1,
  nodo/.style={circle,draw=black!120,fill=white!120,
               inner sep=0pt,minimum size=3mm}
]

\def\n{11}
\def\r{2}

\foreach \i in {1,...,11}
\node[nodo] (\i) at ({90 - (\i-1)*360/\n}:\r) {};

\foreach \i in {1,...,11}{
  \pgfmathtruncatemacro{\j}{int(mod(\i,\n)+1)}
  \ifnum\i<\j \draw (\i)--(\j); \fi

  \pgfmathtruncatemacro{\k}{int(mod(\i+2,\n)+1)}
  \ifnum\i<\k \draw (\i)--(\k); \fi
}
\draw (11)--(1);
\draw (9)--(1);
\draw (10)--(2);
\draw (11)--(3);
\end{tikzpicture}
\qquad\quad
%
\begin{tikzpicture}[
  scale=1,
  nodo/.style={circle,draw=black!120,fill=white!120,
               inner sep=0pt,minimum size=3mm}
]

\def\n{12}
\def\r{2}

\foreach \i in {1,...,12}
\node[nodo] (\i) at ({90 - (\i-1)*360/\n}:\r) {};

\foreach \i in {1,...,12}{

  \pgfmathtruncatemacro{\a}{int(mod(\i,\n)+1)}
  \ifnum\i<\a \draw (\i)--(\a); \fi

  \pgfmathtruncatemacro{\b}{int(mod(\i+5,\n)+1)}
  \ifnum\i<\b \draw (\i)--(\b); \fi

  \ifodd\i
    \pgfmathtruncatemacro{\c}{int(mod(\i+2,\n)+1)}
    \ifnum\i<\c \draw (\i)--(\c); \fi
  \else
    \pgfmathtruncatemacro{\c}{int(mod(\i-4,\n)+1)}
    \ifnum\i<\c \draw (\i)--(\c); \fi
  \fi
}
\draw (12)--(1);
\draw (11)--(2);
\end{tikzpicture}
\qquad\quad
\begin{tikzpicture}[
  scale=1,
  nodo/.style={circle,draw=black!120,fill=white!120,
               inner sep=0pt,minimum size=3mm}
]

\def\r{2}

\foreach \i in {1,...,12}
\node[nodo] (\i) at ({90 - (\i-1)*30}:\r) {};

\node[nodo] (c) at (0,0) {};

\draw (1)--(2)--(3)--(4)--(5)--(6)--(7)--(8)--(9)--(10)--(11)--(12)--(1);

\draw (1)--(4);
\draw (2)--(5);
\draw (3)--(6);
\draw (4)--(7);
\draw[very thick] (5)--(8);     
\draw (6)--(9);
\draw[very thick] (7)--(10);    
\draw (8)--(11);    
\draw (9)--(12);
\draw (10)--(1);
\draw (11)--(2);
\draw[very thick] (12)--(3);    

\draw[very thick] (3)--(c);
\draw[very thick] (5)--(c);
\draw[very thick] (7)--(c);
\draw[very thick] (8)--(c);
\draw[very thick] (10)--(c);
\draw[very thick] (12)--(c);

\draw (1)--(6);
\draw (1)--(8);
\draw (2)--(7);
\draw (2)--(9);
\draw[very thick] (3)--(10);
\draw (4)--(9);
\draw (4)--(11);
\draw[very thick] (5)--(12);
\draw (6)--(11);

\draw[very thick] (7)--(8);

\end{tikzpicture}

}
\end{center}
\caption{Graphs $S(11,4)$ (left), $S(12,4)$ (center) and $S(13,6)$ (right).\label{fig:sporadic}}
\end{figure}

\section{Main results}\label{section:main}
The main results of this section are Proposition~\ref{proposition:conj} which shows that Conjecture~\ref{conjecture:PR} proposed by Petingi and Rodríguez is true in an asymptotic sense, and Theorem~\ref{theorem:summary} which gives a methodology to construct 
infinitely many trace-minimal graphs that are $t$-optimal.\\ 

We now study Conjecture~\ref{conjecture:PR}. The key is Lemma~\ref{lemma:component2}. Its proof is inspired by the proof of Lemma~\ref{lemma:test} that appeared in~\cite{Abrego-2006}. Recall that if $d$ and $n$ are positive integers such that $d\leq n-1$ and $dn$ is even then Harary graph $H_{d}(n)$ is a $d$-regular graph on $n$ vertices thus $\mathcal{R}_d(n)$ is nonempty; a construction of $H_{d}(n)$ is given in~\cite{Harary-1962}. A consequence is Remark~\ref{remark:harary}.
\begin{remark}\label{remark:harary}
If $\mathcal{R}_d(n)$ is nonempty and we let $q$ and $r$ be the only integers such that $n=q(d+1)+r$ and $r\in \{0,\ldots,d\}$, then $\mathcal{R}_{d}(d+1+r)$ is nonempty. 
In fact, as $\mathcal{R}_d(n)$ is nonempty, $d\leq n-1$, and handshaking lemma gives that $dn$ is even. As $nd=qd(d+1)+dr$ 
and $qd(d+1)$ is even, 
we also know that $dr$ is even thus $d(d+1+r)$ is even and Harary graph $H_d(d+1+r)$ is included in $\mathcal{R}_{d}(d+1+r)$.
\end{remark}

\begin{lemma}\label{lemma:component2}
For each pair of integers $d$ and $n$ such that $d\geq 3$ and $n\geq d^3$, every 
$\tau$-max graph in $\mathcal{R}_d(n)$ contains $K_{d+1}$ as a component.
\end{lemma}
\begin{proof}
Let $G$ be any $\tau$-max graph in $\mathcal{R}_d(n)$. Suppose that $G$ does not contain $K_{d+1}$ as a component. 
By Lemma~\ref{lemma:component}, for every $v$ in $V(G)$ we have that $\tau(G,v)\leq \binom{d}{2}-2$. Then,

\begin{equation*}
\tau(G) = \frac{1}{3}\sum_{v\in V(G)} \tau(G,v) \leq \frac{n}{3}\left(\binom{d}{2}-2\right).
\end{equation*}
 Let $q$ and $r$ be integers such that $n = q(d+1) + r$, where $r \in \{0,\ldots,d\}$. By Remark~\ref{remark:harary}, $\mathcal{R}_d(d+1+r)$ is nonempty. Pick any graph $H \in \mathcal{R}_d(d+1+r)$. 
 Define $G' = (q-1)K_{d+1} \cup H$. Then $G' \in \mathcal{R}_d(n)$. As the number of triangles is additive over components,
\begin{equation*}
\tau(G') = (q-1)\tau(K_{d+1})+\tau(H)=
(q-1)\binom{d+1}{3} + \tau(H) \geq (q-1)\binom{d+1}{3}.
\end{equation*}
We claim that $\tau(G') > \tau(G)$. It suffices to show that
\begin{equation*}
(q-1)\binom{d+1}{3} > \frac{n}{3}\left(\binom{d}{2}-2\right).    
\end{equation*}
Using $n = q(d+1)+r$, we obtain that
\begin{align*}
\frac{n}{3}\binom{d}{2} - (q-1)\binom{d+1}{3}
&= \frac{1}{3}(n-(q-1)(d+1))\binom{d}{2} 
= \frac{1}{3}(d+1+r)\binom{d}{2}.
\end{align*}
As $r \leq d$, we have that $d+1+r \leq 2d+1$, and therefore
\[
\frac{1}{3}(d+1+r)\binom{d}{2} \leq \frac{1}{3}(2d+1)\binom{d}{2} = \frac{2d^3-d^2-d}{6} < \frac{d^3}{3}.
\]
As $n \geq d^3$, it follows that 
$\frac{d^3}{3} \leq \frac{n}{3} < \frac{2n}{3}$ hence,
\[
\frac{n}{3}\binom{d}{2} - (q-1)\binom{d+1}{3} < \frac{2n}{3},
\]
which proves the claim. Therefore, $\tau(G')>\tau(G)$, which is impossible since $G$ is $\tau$-max in $\mathcal{R}_d(n)$. Therefore, $G$ includes $K_{d+1}$ as one of its components, and the lemma follows.
\end{proof}




Recall that for each regular graph $G_0$ in $\mathcal{R}_d(n)$ we denote $G_0(p)$ the graph  
$G_0 \cup pK_{d+1}$. We are now in a position to prove Proposition~\ref{proposition:conj}.

\begin{proposition}\label{proposition:conj}
Let $d$ and $n$ be integers such that $d\geq 3$ and $n\geq d^3$. 
Let $G_0$ be in $\mathcal{R}_d(n)$. If $G_0$ is $\nu$-min in $\mathcal{R}_d(n)$ then, for each positive integer $p$, the graph $G_0(p)$ is $\nu$-min in $\mathcal{R}(G_0(p))$. Conversely, if some graph $G'$ is $\nu$-min in $\mathcal{R}(G_0(p))$ then there exists $G_0'$ in $\mathcal{R}_d(n)$ such that $G'=G_0'(p)$.
\end{proposition}
\begin{proof}
Let $d$ and $n$ be any pair of integers such that $d\geq 3$ and $n\geq d^3$. Let $G_0$ be any $\nu$-min 
graph in $\mathcal{R}_d(n)$. By Remark~\ref{remark:counting}, it follows  that $G_0$ is $\tau$-max in $\mathcal{R}_d(n)$. To prove the first part of the statement it suffices to prove that for each positive integer $p$ the graph $G_0(p)$ is $\tau$-max in $\mathcal{R}(G_0(p))$. We will employ induction in the positive integer $p$. First, assume that some graph $G$ is $\tau$-max in the class $\mathcal{R}(G_0(1))$. 
As $G$ is a $d$-regular graph on $n+d+1$ vertices and $n+d+1\geq n\geq d^3$,  Lemma~\ref{lemma:component2} gives that $K_{d+1}$ is a component of $G$. But then, $G=H\cup K_{d+1}$. As $G$ is $\tau$-max and the number of triangles is additive over components, $\tau(G)=\tau(H)+\tau(K_{d+1})$, 
thus the graph $H$ must be $\tau$-max in $\mathcal{R}_d(n)$. As both graphs $H$ and $G_0$ are $\tau$-max in $\mathcal{R}_d(n)$ we get that 
$\tau(H)=\tau(G_0)$. Additionally,  
$\tau(G_0(1)) = \tau(G_0 \cup K_{d+1}) = \tau(G_0)+\tau(K_{d+1})=\tau(H)+\tau(K_{d+1})=\tau(G)$, thus $G_0(1)$ is $\tau$-max in $\mathcal{R}(G_0(1))$ and the base step is proved. 
The inductive step follows by repeating the same argument. Indeed, let $p \geq 1$ and assume that $G_0(p)$ is $\tau$-max in $\mathcal{R}(G_0(p))$. Let $G$ be a $\tau$-max graph in $\mathcal{R}(G_0(p+1))$. As $G$ is $d$-regular on $n+(p+1)(d+1)$ vertices and $n+(p+1)(d+1) \geq n \geq d^3$, Lemma~\ref{lemma:component2} implies that $K_{d+1}$ is a component of $G$. Thus, $G = H \cup K_{d+1}$, where $H \in \mathcal{R}(G_0(p))$. As $G$ is $\tau$-max, the graph $H$ must be $\tau$-max in $\mathcal{R}(G_0(p))$ as well, hence $\tau(H)=\tau(G_0(p))$. Therefore,
\[
\tau(G) = \tau(H)+\tau(K_{d+1}) = \tau(G_0(p))+\tau(K_{d+1}) = \tau(G_0(p+1)),
\]
which proves that $G_0(p+1)$ is $\tau$-max.

Now, let us prove the second part of the statement. By Remark~\ref{remark:counting} it suffices to prove that if some graph $G'$ is $\tau$-max in $\mathcal{R}(G_0(p))$ then there exists $G_0'$ in $\mathcal{R}_d(n)$ such that $G'=G_0'(p)$. Observe that for every $i$ 
in $\{0,\ldots,p\}$ it holds that the graph 
$G_0(i)$ has $n+i(d+1)\geq n \geq d^3$ 
vertices. Repeated application of Lemma~\ref{lemma:component2} gives that $G'=pK_{d+1}\cup H$ for some $\tau$-max graph $H$ in $\mathcal{R}_d(n)$. The proposition follows.
\end{proof}

Now, let us present an enhanced methodology to construct $t$-optimal graphs. 
First, let us find bounds for $t(\overline{G})$ that are sharper than the one given in Lemma~\ref{lemma:PetingiRodriguez}.

\begin{lemma}\label{lemma:improvedbound}
Let $G$ be any graph in $\mathcal{S}_{n,m}$ and let $c$ be any positive integer. For all $x\geq n$,
\begin{equation}\label{equation:polynomial}
P_G(x) \leq x^n e^{-\sum_{k=1}^{c}\frac{g_k(G)}{kx^k}} f(d(G),x).  
\end{equation}
Additionally, $t(\overline{G}) \leq n^{n-2}e^{-\sum_{k=1}^{c}\frac{g_k(G)}{kn^k}}f(d(G),n)$, where $f$ is defined in expression~\eqref{eq:f}. 
\end{lemma}
\begin{proof}
Let $G$ be any graph $\mathcal{S}_{n,m}$ and let $c$ be any positive integer. By expression~\eqref{eq:nodrop} and Remark~\ref{remark:key}, for all $x\geq n$, 
\begin{align*}
-\log\left(\frac{P_G(x)}{x^n}\right) &=\sum_{k=1}^{\infty}\frac{\ell_k(G)}{kx^k} \geq 
\sum_{k=1}^{\infty}\frac{\sum_{i=1}^{n}d_i(d_i+1)^{k-1}}{kx^k} + \sum_{k=1}^{c}\frac{g_k(G)}{kx^k} \\
&= -\log\left(e^{-\sum_{k=1}^{c}\frac{g_k(G)}{kx^k}}\prod_{i=1}^{n}\left(1-\frac{d_i+1}{x} \right)^{\frac{d_i}{d_i+1}} \right),
\end{align*}
where the last equality follows from~\eqref{ineq:chain}. 
Solving for $P_G(x)$ yields Equation~\eqref{equation:polynomial}. 
The second part of the statement follows from Lemma~\ref{lemma:t}.
\end{proof}

Define $\mathcal{S}_{n,m}^{(1)}$ as $\mathcal{S}_{n,m}$. For each positive integer $k$, define $\mathcal{S}_{n,m}^{(k+1)}$ as follows,
\begin{equation*}
\mathcal{S}_{n,m}^{(k+1)} = \{G: G\in \mathcal{S}_{n,m}^{(k)}, \,\,\ell_{k+1}(G)\leq \ell_{k+1}(H) \text{ for each } H \text{ in }  \mathcal{S}_{n,m}^{(k)}\}.    
\end{equation*}
Recall that $\ell_2(G)=\sum_{i=1}^{n}d_i(d_i+1)$. Then, $\ell_2(G)$ is minimized among all graphs $G$ in $\mathcal{S}_{n,m}$ precisely when $G$ is almost-regular thus 
$\mathcal{S}_{n,m}^{(2)}=\mathcal{A}_{n,m}$. As 
$\ell_3(G)=\sum_{i=1}^{n}d_i(d_i+1)^2 + 2\nu(G)$, the expression $\ell_3(G)$ is minimized among almost-regular graphs precisely when $G$ is $\nu$-min. As a consequence, $\mathcal{S}_{n,m}^{(3)}$ is the set consisting of all $\nu$-min graphs in $\mathcal{A}_{n,m}$.  
In Proposition~\ref{proposition:generalize} we will generalize Proposition~\ref{proposition:PetingiRodriguez}. First, let us prove two technical lemmas.
\begin{lemma}\label{lemma:gmh}
Let $G_0$ be an almost-regular graph whose vertices have degree $d-1$ or $d$. Let 
$\lambda_1,\lambda_2,\ldots,\lambda_n$ be the eigenvalues of \pr{$L(G_0)$}. Let $p$ and $q$ be nonnegative integers. Define the function $g(x)$ as follows,
\begin{equation}\label{eq:check}
g(x) = \left(1-\frac{d+1}{x}\right)^{pd}\left( 1-\frac{d}{x}\right)^{q(d-1)}\prod_{i=1}^{n}\left(1-\frac{\lambda_i}{x}\right).    
\end{equation}    
Let $n'$ be the number of vertices in $G_0(p,q)$, i.e., $n'=n(G_0(p,q))$. Then, the following assertions hold.
\begin{enumerate}[label=(\roman*)]
\item\label{gmh2} $t(\overline{G_0(p,q)}) = (n')^{n'-2}g(n')$.
\item\label{gmh3} The gap sequences of $G_0(p,q)$ and $G_0$ are identical.
\end{enumerate}
\end{lemma}
\begin{proof}
By Remark~\ref{remark:complete}, it follows that $P_{K_{d+1}}(x)=x(x-(d+1))^d$ and $P_{K_{d}}(x)=x(x-d)^{d-1}$. Recall that $G_0(p,q)=G_0\cup pK_{d+1}\cup qK_d$. As the Laplacian polynomial factors over components, we obtain that  
\begin{align*}
P_{G_0(p,q)}(x) &= P_{K_{d+1}}(x)^p P_{K_d}(x)^q P_{G_0}(x) 
= (x(x-(d+1))^d)^p (x(x-d)^{d-1})^q \prod_{i=1}^{n}(x-\lambda_i)\\ 
&= x^{p+q}(x-(d+1))^{pd}(x-d)^{q(d-1)}\prod_{i=1}^{n}(x-\lambda_i) = x^{p+q}x^{pd}x^{q(d-1)}\pr{x^n}g(x)=x^{n'}g(x),
\end{align*}
where we used the definition of $g(x)$ and the fact that 
$n'=n(G)=n+p(d+1)+qd$ in the last steps. 
Consequently, $P_{G_0(p,q)}(x)=x^{n'}g(x)$. Applying Lemma~\ref{lemma:t} to the graph $G_0(p,q)$ gives that 
$t(\overline{G_0(p,q)})=(n')^{-2}P_{G_0(p,q)}(n')=(n')^{n'-2}g(n')$ thus proving~\ref{gmh2}. Finally, let $k$ be any positive integer. As $g_k$ is additive over components, 
\begin{equation*}
g_k(G_0(p,q))=g_k(G_0\cup pK_{d+1}\cup qK_d)=
g_k(G_0)+pg_k(K_{d+1})+qg_k(K_d)= g_k(G_0),    
\end{equation*}
where we used Lemma~\ref{lemma:gaps} in the last equality 
which states that the gap sequence of any complete graph is $0$ thus $g_k(K_d)=g_k(K_{d+1})=0$. 
Consequently, $g_k(G_0(p,q))=g_k(G_0)$ thus proving~\ref{gmh3}. The lemma follows.
\end{proof}

\begin{lemma}\label{lemma:boundgap}
Let $G_0$ be in $\mathcal{A}_{n,m}$ whose vertices have degree $d-1$ or $d$. 
For each positive integer $k$, $g_k(G_0)\leq n (2d)^{k}$.
\end{lemma}
\begin{proof}
Let $k$ and $G_0$ be as in the statement. Let $\lambda_1,\lambda_2,\ldots,\lambda_n$ be the eigenvalues of $L(G_0)$. Let $i\in \{1,2,\ldots,n\}$. By Lemma~\ref{lemma:basics}\ref{prop4}, we  know that $\lambda_i$ lies in the interval $[0,2d]$. On the one hand, $\ell_k(G_0)=\sum_{i=1}^{n}\lambda_i^k\leq n(2d)^k$. On the other hand, 
$\sum_{i=1}^{n}d_i(d_i+1)^{k-1}\geq 0$. Consequently, 
$g_k(G_0)=\ell_k(G_0)-\sum_{i=1}^{n}d_i(d_i+1)^{k-1} \leq n(2d)^k$, 
and the lemma follows.
\end{proof}

\begin{proposition}\label{proposition:generalize}
Let $c$ be any positive integer. Let $G_0$ be an almost-regular graph whose vertices have degree $d-1$ or $d$.
Assume that for each pair of nonnegative integers $p$ and $q$ it holds that $G_0(p,q)\in \mathcal{S}(G_0(p,q))^{(c+1)}$. Then, for any graph $H$ in $\mathcal{S}(G_0(p,q))$ not in $\mathcal{S}(G_0(p,q))^{(c+1)}$ with at least $2d+n(G_0)(2d)^{c+2}$ vertices it holds that
$t(\overline{G_0(p,q)})>t(\overline{H})$.
\end{proposition}

\begin{proof}
Let $c$ and $G_0$ be as in the statement. If $c=1$, the result follows from Proposition~\ref{proposition:PetingiRodriguez-almostregular} and Remark~\ref{remark:cobertura}. Assume $c\geq 2$. Let $r$ and $s$ be the number of vertices in $G_0$ with degrees $d-1$ and $d$, respectively, and define
\begin{equation*}
h(x,y) = e^{-\sum_{k=1}^c\frac{g_k(G_0)}{kx^k}-\frac{y}{(c+1)x^{c+1}}}\left(1-\frac{d+1}{x}\right)^{(s+p(d+1))d/(d+1)}\left(1-\frac{d}{x} \right)^{(r+qd)(d-1)/d}.
\end{equation*}

We claim that if $g'>g_{c+1}(G_0)$  then $
h(x,g')<g(x)$ for all $x\geq 2d+n(G_0)(2d)^{c+2}$.\\

We first prove that the claim implies the proposition. In fact, if the claim holds then let $H\in \mathcal{S}(G_0(p,q))-\mathcal{S}(G_0(p,q))^{(c+1)}$. Without loss of generality, we may assume that $H\in \mathcal{S}(G_0(p,q))^{(c)}-\mathcal{S}(G_0(p,q))^{(c+1)}$. Let $k\in \{1,\ldots,c\}$. Lemma~\ref{lemma:gmh}\ref{gmh3} gives that  $g_k(G_0(p,q))=g_k(G_0)$. Since $c\geq 2$, the graph $H$ is almost-regular and hence $d(H)=d(G_0(p,q))$. Therefore,
\[
g_k(H)-g_k(G_0(p,q))=\ell_k(H)-\ell_k(G_0(p,q))=\ell_k(H)-\ell_k(G_0).
\]
As $H\in \mathcal{S}(G_0(p,q))^{(c)}$, we have $\ell_k(H)=\ell_k(G_0)$, and thus $g_k(H)=g_k(G_0)$.

Let $n'=n(G_0(p,q))$ and $g'=g_{c+1}(H)$. As $H\notin \mathcal{S}(G_0(p,q))^{(c+1)}$, we have $\ell_{c+1}(H)>\ell_{c+1}(G_0)$, and hence $g'>g_{c+1}(G_0)$. On the one hand, Lemma~\ref{lemma:improvedbound} gives that
\begin{equation}\label{eq:one}
t(\overline{H})\leq (n')^{n'-2}h(n',g').    
\end{equation}
On the other hand, Lemma~\ref{lemma:gmh}\ref{gmh2} gives that
\begin{equation}\label{eq:other}
t(\overline{G_0(p,q)})=(n')^{n'-2}g(n').    
\end{equation}
If we assume that the claim holds, then 
$h(x,g')<g(x)$ when $x\geq 2d+n(G_0)(2d)^{c+2}$, and by expressions~\eqref{eq:one} and \eqref{eq:other} we get that 
\[
t(\overline{H}) \leq (n')^{n'-2}h(n',g') < (n')^{n'-2} g(n') = t(\overline{G_0(p,q)}),
\]
which proves the proposition.

It remains to prove the inequality $h(x,g')<g(x)$ whenever $x\geq 2d+n(G_0)(2d)^{c+2}$. 
As $x>2d$, both $g(x)$ and $h(x,g')$ are positive. Define $w(x)=g(x)/h(x,g')$. Then
\begin{align*}
\log w(x) &= \sum_{i=1}^{n}\log\left(1-\frac{\lambda_i}{x}\right)
-\frac{s d}{d+1}\log\left(1-\frac{d+1}{x}\right)
-\frac{r(d-1)}{d}\log\left(1-\frac{d}{x}\right) 
+ \sum_{i=1}^{c}\frac{g_i(G_0)}{i x^i}+
 \frac{g'}{(c+1)x^{c+1}}.
  \end{align*}

Now, taking the derivative of $\log(w(x))$ with respect to $x$ and using the convergence of geometric series followed by the definition of the gap sequence of $G_0$, we obtain that
\begin{align*}
\log(w(x))' &= \sum_{i=1}^{n}\frac{\lambda_i}{x^2(1-\lambda_i/x)}-\frac{sd}{x^2}\frac{1}{1-\frac{d+1}{x}}-\frac{r(d-1)}{x^2}\frac{1}{1-\frac{d}{x}}-\sum_{i=1}^{c}\frac{g_i(G_0)}{x^{i+1}}-\frac{g'}{x^{c+2}}\\
&= \frac{1}{x^2} \left(\sum_{j=0}^{\infty}\frac{(\sum_{i=1}^{n}\lambda_i^{j+1})-sd(d+1)^j-r(d-1)d^j}{x^j}   -\sum_{j=1}^{c}\frac{g_j(G_0)}{x^{j-1}}-\frac{g'}{x^c}\right)\\
&=\frac{1}{x^2} \left(\sum_{j=0}^{\infty}\frac{\ell_{j+1}(G_0)-\sum_{i=1}^{n}d_i(G_0)(d_i(G_0)+1)^{j}}{x^j} -\sum_{j=0}^{c-1}\frac{g_{j+1}(G_0)}{x^{j}}-\frac{g'}{x^c}\right)\\
&= \sum_{j=0}^{\infty}\frac{g_{j+1}(G_0)}{x^{j+2}}- \sum_{j=0}^{c-1}\frac{g_{j+1}(G_0)}{x^{j+2}}-\frac{g'}{x^{c+2}}=\frac{g_{c+1}(G_0)-g'}{x^{c+2}}+\sum_{k=c+3}^{\infty}\frac{g_{k-1}(G_0)}{x^k}\\
&\leq  \frac{g_{c+1}(G_0)-g'}{x^{c+2}}+\sum_{k=c+3}^{\infty}\frac{n(G_0)(2d)^{k-1}}{x^k}\\
&=\frac{g_{c+1}(G_0)-g'}{x^{c+2}} + \frac{n(G_0)}{2d}\frac{(2d)^{c+3}}{x^{c+3}}\frac{1}{1-(2dx^{-1})}\\ 
&\leq \frac{1}{x^{c+2}}\left(\frac{n(G_0)(2d)^{c+2}}{x-2d}-1\right),
\end{align*}
where the first inequality uses Lemma~\ref{lemma:boundgap} and the second inequality uses that $g'-g_{c+1}(G_0)\geq 1$ since both $g'$ and $g_{c+1}(G_0)$ are integers and $g'>g_{c+1}(G_0)$. Hence $
\log (w(x))'<0$ when $x\geq 2d+n(G_0)(2d)^{c+2}$. Since $w(x)\to 1$ as $x\to\infty$, it follows that $w(x)>1$ when $x\geq 2d+n(G_0)(2d)^{c+2}$, and therefore $h(x,g')<g(x)$. This completes the proof.
\end{proof}

Proposition~\ref{proposition:PetingiRodriguez} establishes that the complement of a graph $G_0(p,q)$ with a large number of vertices is $t$-optimal only if $G_0(p,q)$ is an almost-regular $\nu$-min graph. Proposition~\ref{proposition:generalize} sharpens this result by showing that, within the same class, the complement of $G_0(p,q)$ is lexicographically minimal with respect to the Laplacian sequence. This observation motivates the introduction of the following concept, which is analogous to that of trace-minimal graphs.

\begin{definition}
A graph $G$ in $\mathcal{R}_{d}(n)$ is \emph{$\mathcal{L}$-trace-minimal} if for each $H$ in $\mathcal{R}_{d}(n)$ it holds that $(\ell_i(G)) \preceq (\ell_i(H))$. 
\end{definition}

In Lemma~\ref{lemma:t-unique} we show how to construct $t$-optimal graphs finding $\mathcal{L}$-trace minimal graphs. 
The following remark will be used.
\begin{remark}\label{remark:ltracetriangle}
Each $\mathcal{L}$-trace minimal graph 
in $\mathcal{R}_d(n)$ is $\tau$-max. 
In fact, by its definition we know that 
if $G$ is \pr{a} $\mathcal{L}$-trace minimal graph then $\ell_3(G)$ is minimum, hence 
$G$ is $\nu$-min or $\tau$-max by Remark~\ref{remark:counting}.
\end{remark}

\begin{lemma}\label{lemma:t-unique}
Let $n$ and $d$ be integers such that $d\geq 3$ and $n\geq d^3$, and let $G_0$ be in $\mathcal{R}_d(n)$. If $G_0(p)$ is an $\mathcal{L}$-trace-minimal graph in $\mathcal{R}(G_0(p))$ for each nonnegative integer $p$ then there exists $n_0$ such that each graph $\overline{G_0(p)}$ on at least $n_0$ vertices is $t$-optimal.  
\end{lemma}

\begin{proof}
Let $n$ and $d$ be integers such that $d\geq 3$ and $n\geq d^3$. Let $G_0$ be any graph in $\mathcal{R}_d(n)$ such that $G_0(p)$ is an $\mathcal{L}$-trace-minimal graph for each nonnegative integer $p$. Define $n_0=2d+n(2d)^{n+2}$. It suffices to prove that 
$\overline{G_0(p)}$ is $t$-optimal when $n(G_0(p))\geq n_0$. In fact, let $p$ be any integer such that $n(G_0(p))\geq n_0$. Let $H$ be any graph in $\mathcal{S}(G_0(p))$. The graph $H$ satisfies precisely one of the following conditions: (i) $H$ is not $d$-regular; (ii) $H$ is $d$-regular but not $\tau$-max; 
(iii) $H$ is $d$-regular $\tau$-max but not $\mathcal{L}$-trace-minimal; or (iv) $H$ is 
$d$-regular $\tau$-max and $\mathcal{L}$-trace-minimal. If condition (iv) is satisfied then by Remark~\ref{remark:ltracetriangle} the graph $H$ is just $\mathcal{L}$-trace minimal. 
Let us prove that $t(\overline{G_0(p)})\geq t(\overline{H})$ in each of the cases separately.
\begin{enumerate}[label=(\roman*)]
\item[(i)] If $H$ is not $d$-regular then observe that $n_0 \geq 2d+n(2d)^{3}$. 
By Remark~\ref{remark:cobertura} we know that $n_0$ satisfies the conditions of Proposition~\ref{proposition:PetingiRodriguez-almostregular}. By Proposition~\ref{proposition:PetingiRodriguez-almostregular},  
$t(\overline{G_0(p)})>t(\overline{H})$. 
\item[(ii)] If $H$ is $d$-regular but not $\tau$-max, then 
by Remark~\ref{remark:counting} we know that $H$ 
is $d$-regular but not $\nu$-min thus $H$ is in $\mathcal{S}(G_0(p))^{(2)}-\mathcal{S}(G_0(p))^{(3)}$. Observe that the inequality $n_0\geq 2d+n(2d)^{4}$ is satisfied. Then, Proposition~\ref{proposition:generalize} gives that $t(\overline{G_0(p)})>t(\overline{H})$. 
\item[(iii)] If $H$ is $d$-regular $\tau$-max but it is not $\mathcal{L}$-trace-minimal then by Proposition~\ref{proposition:conj} it follows that $H=H_0(p)$. As both graphs $H_0$ and $G_0$ are $d$-regular, $\ell_k(H_0(p))-\ell_k(G_0(p))=\ell_k(H_0)-\ell_k(G_0)$ for each positive integer $k$. 
Observe that if $H_0$ were $\mathcal{L}$-trace-minimal then the graphs $G_0$ and $H_0$ would have identical Laplacian spectrum as well as $G_0(p)$ and $H_0(p)$, and $H_0(p)$ (i.e., $H$) would be an $\mathcal{L}$-trace-minimal graph in contradiction with our assumption. Consequently, there exists $c$ in $\{1,\ldots,n-1\}$ such that 
$H_0$ is in $\mathcal{S}(G_0)^{(c)}$ but not in $\mathcal{S}(G_0)^{(c+1)}$. 
This means that $\ell_{c+1}(H_0)>\ell_{c+1}(G_0)$, which implies that $\ell_{c+1}(H_0(p))>\ell_{c+1}(G_0(p))$ hence $H_0(p)$ is in $\mathcal{S}(G_0(p))^{(c)}$ but not in $\mathcal{S}(G_0(p))^{(c+1)}$. As $c<n$ we know that $n_0=2d+n(2d)^{n+2}>2d+n(2d)^{c+2}$. By Proposition~\ref{proposition:generalize} 
we conclude that $t(\overline{G_0(p)})>t(\overline{H_0(p)})=t(\overline{H})$.
\item[(iv)] If $H$ is $\mathcal{L}$-trace-minimal as well as $G_0(p)$ then its complements $\overline{H}$ and $G_0(p)$ 
share the same Laplacian spectrum and $t(\overline{G_0(p)})=t(\overline{H})$. 
\end{enumerate} 
In either case, $t(\overline{G_0(p)})\geq t(\overline{H})$. 
As $H$ is an arbitrary graph in $\mathcal{S}(G_0(p))$, we conclude that 
$\overline{G_0(p)}$ is $t$-optimal in $\mathcal{S}(\overline{G_0(p)})$ when $n(G_0(p))\geq n_0$, and the lemma follows.  
\end{proof}
To apply Lemma~\ref{lemma:t-unique} we need to find $\mathcal{L}$-trace minimal graphs. As it will turn out, the key is a strong relationship between $\mathcal{L}$-trace-minimal and trace-minimal graphs. 
In fact, we will prove that a graph $G$ is $\mathcal{L}$-trace-minimal in $\mathcal{R}_d(n)$ if and only if $\overline{G}$ is trace-minimal in $\mathcal{R}_{n-1-d}(n)$. As a consequence, all the mathematical properties of trace-minimal graphs will be useful to find $\mathcal{L}$-trace minimal graph and, by Lemma~\ref{lemma:t-unique}, new $t$-optimal regular graphs. We will use Lemma~\ref{lemma:algebra} whose proof includes basic concepts of linear algebra. 
\begin{lemma}\label{lemma:algebra}
For each $G$ in $\mathcal{R}_d(n)$ and each pair of positive integers $i$ and $j$ it holds that $\mathrm{tr}(L(G)^i((d+1)I_n-J_n)^j)=(d+1)^j\ell_i(G)$.
\end{lemma}
\begin{proof}
Let $i$ and $j$ be any positive integers. Let us find $\mathrm{tr}(L(G)^i((d+1)I_n-J_n)^j)$. 
As the sum of each row in $L(G)$ equals $0$ we get that $L(G)J_n$ is the all-zeroes matrix $0_n$ and $\mathrm{tr}(L(G)^iJ_n)=\mathrm{tr}(L(G)^{i-1}0_n)=0$.
As $J_n^2=nJ_n$, by induction it follows that $J_n^j=n^{j-1}J_n$ for any positive integer $j$. By the linearity of the trace operator, 
\begin{align*}
\mathrm{tr}(L(G)^i((d+1)I_n-J_n)^j) &= 
\mathrm{tr}(L(G)^i \sum_{k=0}^{j}\binom{j}{k}((d+1)I_n)^k(-J_n)^{j-k})\\
&=\sum_{k=0}^{j-1}\binom{j}{k}(d+1)^k(-1)^{j-k}n^{j-k-1}\mathrm{tr}(L(G)^iJ_n)+ \mathrm{tr}((L(G)^i (d+1)^jI_n))\\
&= \mathrm{tr}((L(G)^i (d+1)^jI_n))=(d+1)^j\ell_i(G). \qedhere
\end{align*}
\end{proof}

\begin{lemma}\label{lemma:lminimal}
A graph $G$ is $\mathcal{L}$-trace-minimal in $\mathcal{R}_d(n)$ if and only if $\overline{G}$ is trace-minimal in $\mathcal{R}_{n-1-d}(n)$.    
\end{lemma}
\begin{proof}
We will prove the direct since a proof of the converse is analogous. Let $G$ be any $\mathcal{L}$-trace-minimal graph in $\mathcal{R}_d(n)$. As $G$ is $\mathcal{L}$-trace-minimal, for each $H$ in $\mathcal{R}_d(n)$ precisely one of the following two conditions holds:
\begin{enumerate}[label=(\roman*)]
\item\label{c1} For each positive integer $k$ it holds that $\ell_k(G)=\ell_k(H)$ (i.e., the sequences are identical), or 
\item\label{c2} There exists a positive integer $j$ such that for each $k\in \{1,2,\ldots,j-1\}$ it holds that $\ell_k(G)=\ell_k(H)$ but $\ell_j(G)<\ell_j(H)$.
\end{enumerate}
As $\overline{H}$ is any graph in $\mathcal{R}_{n-1-d}(n)$, it suffices to prove that $(a_i(\overline{G}))\preceq (a_i(\overline{H}))$. Observe that 
$A(\overline{G})+A(G)=J_n-I_n$. As $G$ is in $\mathcal{R}_d(n)$ we know that $L(G)=dI_n-A(G)$. Consequently, 
$A(\overline{G})=J_n-I_n-A(G) =L(G)-((d+1)I_n-J_n)$. Then,
\begin{align*}
a_i(\overline{G})&=\mathrm{tr}(A(\overline{G})^i) = \mathrm{tr}((L(G)-((d+1)I_n-J_n))^i)\\ 
&= \sum_{k=0}^{i}(-1)^{i-k}\binom{i}{k}\mathrm{tr}(L(G)^k ((d+1)I_n-J_n)^{i-k}) =\sum_{k=0}^{i}(-1)^{i-k}\binom{i}{k}(d+1)^{i-k}\ell_k(G),
\end{align*}
where the last step holds by Lemma~\ref{lemma:algebra}. Similarly, as $H$ is in 
$\mathcal{R}_d(n)$ we also know that $A(\overline{H})=L(H)-((d+1)I_n-J_n)$, and the same calculation for $a_i(\overline{H})$ gives that
\begin{equation*}
a_i(\overline{H}) = \sum_{k=0}^{i}(-1)^{i-k}\binom{i}{k}(d+1)^{i-k}\ell_k(H).
\end{equation*}

As a consequence, for every positive integer $i$ we obtain that 
\begin{equation}\label{eq:comparethem}
a_i(\overline{G})-a_i(\overline{H}) = 
\sum_{k=0}^{i}(-1)^{i-k}\binom{i}{k}(d+1)^{i-k}(\ell_k(G)-\ell_k(H)).
\end{equation}

If condition~\ref{c1} holds then the sequences $\ell_k(G)$ and $\ell_k(H)$ are identical. Then, by  Equation~\eqref{eq:comparethem}, the sequences $a_i(\overline{G})$ and $a_i(\overline{H})$ are also identical. Finally, if condition~\ref{c2} holds then 
there exists a positive integer $j$ such that for each $k\in \{1,2,\ldots,j-1\}$ it holds that 
$\ell_k(G)=\ell_k(H)$ but $\ell_j(G)<\ell_j(H)$. In this case Equation~\eqref{eq:comparethem} gives that 
$a_i(\overline{G})=a_i(\overline{H})$ when  
$i\in \{1,2,\ldots,j-1\}$ but $a_j(\overline{G})<a_j(\overline{H})$. In either case, $(a_i(\overline{G})) \preceq (a_i(\overline{H}))$ and $\overline{G}$ is trace-minimal in $\mathcal{R}_{n-1-d}(n)$. The lemma follows.  
\end{proof}

Lemma~\ref{lemma:lminimal} establishes a correspondence between $\mathcal{L}$-trace-minimal graphs in $\mathcal{R}_d(n)$ and trace-minimal graphs in $\mathcal{R}_{n-1-d}(n)$. As a consequence, the problem of identifying $\mathcal{L}$-trace-minimal graphs can be reduced to the study of trace-minimal graphs, for which several structural results and constructions are already available in the literature. 

Our findings are summarized in Theorem~\ref{theorem:summary}.

\begin{theorem}\label{theorem:summary}
For each pair of integers $d$ and $n$ such that $d\geq 3$ and 
$n\geq d^3$ and each graph $G_0$ in $\mathcal{R}_d(n)$ whose complement is trace-minimal, there exists infinitely many nonnegative integers $p$ such that 
$\overline{G_0(p)}$ is both a trace-minimal and a $t$-optimal graph.
\end{theorem}
\begin{proof}
Let $d$ and $n$ be integers such that $d\geq 3$ and $n\geq d^3$, and let $G_0$ be any graph 
in $\mathcal{R}_d(n)$ whose complement is trace-minimal. Let $p$ be any nonnegative integer. By Lemma~\ref{lemma:t-unique}, it suffices to prove that $G_0(p)$ is $\mathcal{L}$-trace minimal. 
Let $H_p$ be any graph with as many vertices and edges as $G_0(p)$ whose complement is trace-minimal.  
As $H_p$ has at least $n$ vertices and $n\geq d^3$, a repeated application of Lemma~\ref{lemma:corollary} 
gives that $\overline{H_p} = E_{d+1}^{(p)} \wedge \overline{H_0}$, or equivalently, 
$H_p = H_0 \cup pK_{d+1}$, where $H_0 \in \mathcal{R}_d(n)$ and its complement is trace-minimal. 
As $\overline{G_0}$ and $\overline{H_0}$ are trace-minimal, both graphs have identical adjacency spectrum. Consequently, $\overline{G_0(p)}$ and $\overline{H_p}$ also have identical adjacency spectrum and $\overline{G_0(p)}$ is also trace-minimal. As $\overline{G_0(p)}$ is trace-minimal, Lemma~\ref{lemma:lminimal} gives that $G_0(p)$ is $\mathcal{L}$-trace-minimal. By Lemma~\ref{lemma:t-unique}, there exists  $p_0$ such that 
$\overline{G_0(p)}$ is $t$-optimal whenever $p\geq p_0$. The theorem is proved.
\end{proof}

We are now in a position to establish a two-step method to find new $t$-optimal graphs. 
First, pick integers $d$ and $n$ such that $d\geq 3$ and $n\geq d^3$. Then, find some graph $G$ in $\mathcal{R}_d(n)$ whose complement is trace-minimal. Theorem~\ref{theorem:summary} asserts that 
there exists a nonnegative integer $p_0$ such that $\overline{G\cup pK_{d+1}}$ is $t$-optimal (and trace-minimal) whenever $p\geq p_0$.

Our method avoids spectral comparisons as well as the determination of infinitely many almost-regular $\nu$-min graphs that appear in~\cite{Petingi-2002}. Moreover, it allows us to rely on existing tools for the construction of trace-minimal graphs available in the literature.

\section{Construction of regular t-optimal graphs}\label{section:toptimal}
A generous list of trace-minimal graphs can be found in~\cite{Abrego-2006}. Among them, we can find all trace-minimal graphs in each nonempty class $\mathcal{R}_{n-1-d}(n)$ such that $d\in \{1,\ldots,5\}$. By Theorem~\ref{theorem:summary}, such classes of trace-minimal graphs are also $t$-optimal when the number of vertices is large enough. Petingi and Rodríguez~\cite{Petingi-2002} found the only $t$-optimal graph in 
$\mathcal{R}_{n-1-d}(n)$ when $d\in \{0,1,2,3\}$ and $n$ is large enough. 
We found the only $t$-optimal graph in $\mathcal{R}_{n-1-d}(n)$ when $d=4$ and $n$ is large enough~\cite{RomeroLAGOS1-2025}. Independently, Gong \textit{et~al.}~\cite{Gong-2026} found the only $t$-optimal graph in $\mathcal{R}_{n-1-d}(n)$ when $d \in \{4,5\}$ and $n$ is large enough. All the aforementioned results now follow from Theorem~\ref{theorem:summary} after revisiting the trace-minimal graphs determined by Ábrego \textit{et~al.}~\cite{Abrego-2006}. In Lemma~\ref{lemma:6} and Lemma~\ref{lemma:7} we find the only trace-minimal graph in each graph class $\mathcal{R}_{n-1-d}(n)$ when $d=6$ and $d=7$, respectively. 

\begin{lemma}\label{lemma:6}
Let $G$ be any trace-minimal graph in $\mathcal{R}_{n-7}(n)$. Let $q$ and $r$ be the only integers such that $r\in \{0,\ldots,6\}$ and $n=7q+r$. All the following assertions hold,
\begin{enumerate}[label=(\roman*)]
\item \label{a0} If $r=0$ then $G=E_{7}^{(q)}$;
\item \label{a1} If $r=1$ then $G=E_{7}^{(q-1)} \wedge (4K_2)$;
\item \label{a2} If $r=2$ then $G=E_{7}^{(q-1)} \wedge C_9$;
\item \label{a3} If $r=3$ then $G=E_{7}^{(q-3)} \wedge S(24,17)$;
\item \label{a4} If $r=4$ then $G=E_{7}^{(q-4)} \wedge S(32,25)$ ;
\item \label{a5} If $r=5$ then $G=E_{7}^{(q-3)} \wedge S(26,19)$;
\item \label{a6} If $r=6$ then $G=E_{7}^{(q-3)} \wedge S(27,20)$.
\end{enumerate}
\end{lemma}
\begin{proof}
Let $G$ be any trace-minimal graph in $\mathcal{R}_{n-7}(n)$. Let $q$ and $r$ be the only integers such that $r\in \{0,\ldots,6\}$ and $n=7q+r$. First, let us find $\tau_r(d+1+r)$. If $d=6$ and $r\in \{1,3,5\}$ then $d+1+r$ is even and the inequality $2r\leq d+1+r$ holds. Then, Lemma~\ref{lemma:trianglefree} gives that  $\tau_r(d+1+r)=0$. If $d=6$ and $r\in \{0,2,4\}$ then $d+1+r \in \{7,9,11\}$. Observe that the edgeless graph $E_7$ and the cycle $C_9$ are triangle-free 
graphs in $\mathcal{R}_{0}(7)$ and $\mathcal{R}_{2}(9)$, respectively. 
Lemma~\ref{lemma:cases} gives that $S(11,4)$ is the only trace-minimal graph in $\mathcal{R}_4(11)$, which is triangle-free. 
By Remark~\ref{remark:triangles}, each trace-minimal graph is $\tau$-min. Therefore, in all previous cases, $\tau_r(d+1+r)=0$, and Equation~\eqref{eq:w} yields $\mu(6,r)=\frac{r}{4}(49-r^2)$. Finally, if $d=6$ and $r=6$ then 
Lemma~\ref{lemma:cases} gives that 
the only trace-minimal graph $S(13,6)$ in $\mathcal{R}_{6}(13)$ has precisely $6$ triangles (see Figure~\ref{fig:sporadic}). As trace-minimal graphs have the least number of triangles we conclude that $\tau_6(13)=6$, and Equation~\eqref{eq:w} yields $\mu(6,6)=\frac{6}{4}(49-6^2)+\frac{3}{2}6 = \frac{57}{2}$. 
Now, let us prove each assertion separately.
\begin{enumerate}[label=(\roman*)]
\item If $r=0$ then $\mu(d,r)=\mu(6,0)=0$. By repeated application of Lemma~\ref{lemma:test}, $G=E_{7}^{(q)}$.
\item If $r=1$ then $\mu(d,r)=\mu(6,1)=12$. By repeated application of Lemma~\ref{lemma:test}, $G=E_{7}^{(q-1)}\wedge S(8,1)$. As 
$\mathcal{R}_{1}(8)=\{4K_2\}$, $4K_2$ is trace-minimal and $S(8,1)=4K_2$.
\item If $r=2$ then $\mu(d,r)=\mu(6,2)=\frac{45}{2}$. By repeated application of Lemma~\ref{lemma:test}, $G=E_{7}^{(q-2)} \wedge S(16,9)$. Observe that $S(16,9)$ is a trace-minimal graph in $\mathcal{R}_{n+1}(2n)$ with $n=8$. By Proposition~\ref{proposition:saegusa} we get that $S(16,9)=E_7 \wedge C_9$ 
and $G=E_{7}^{(q-1)} \wedge C_9$.
\item If $r=3$ then $\mu(d,r)=\mu(6,3)=30$. By repeated application of Lemma~\ref{lemma:test}, $G=E_{7}^{(q-3)} \wedge S(24,17)$.
\item If $r=4$ then $\mu(d,r)=\mu(6,4)=33$. By repeated application of Lemma~\ref{lemma:test}, $G=E_{7}^{(q-4)} \wedge S(32,25)$.
\item If $r=5$ then $\mu(d,r)=\mu(6,5)=30$. 
By Lemma~\ref{lemma:test}, $G=E_{7}^{(q-3)} \wedge S(26,19)$.
\item If $r=6$ then $\mu(d,r)=\mu(6,6)=\frac{57}{2}$. By repeated application of Lemma~\ref{lemma:test}, $G=E_{7}^{(q-3)} \wedge S(27,20)$. \qedhere 
\end{enumerate}
\end{proof}

\begin{lemma}\label{lemma:7}
Let $G$ be any trace-minimal graph in $\mathcal{R}_{n-8}(n)$. Let $q$ and $r$ be the only integers such that $r\in \{0,2,4,6\}$ and $n=8q+r$. All the following assertions hold,
\begin{enumerate}[label=(\roman*)]
\item \label{b0} If $r=0$ then $G=E_{8}^{(q)}$;
\item \label{b2} If $r=2$ then $G=E_{8}^{(q-3)} \wedge S(26,18)$;
\item \label{b4} If $r=4$ then $G=E_{8}^{(q-5)} \wedge S(44,36)$ ;
\item \label{b6} If $r=6$ then $G=E_{8}^{(q-4)} \wedge S(38,30)$.
\end{enumerate}
Additionally, if Conjecture~\ref{conjecture:Saegusa} holds then 
$S(26,18)=E_8^{(2)} \wedge C_{10}$, $S(44,36)=E_8^{(4)} \wedge S(12,4)$ where $S(12,4)$ is given in Figure~\ref{fig:sporadic}, and $S(38,30)=E_8^{(3)} \wedge (K_{7,7}-7K_2)$.
\end{lemma}
\begin{proof}
Let $G$ be any trace-minimal graph in $\mathcal{R}_{n-8}(n)$. By handshaking lemma, $n(n-8)$ is even thus $n$ is even. Let $q$ and $r$ be the only integers such that $r\in \{0,2,4,6\}$ and $n=8q+r$. First, let us find $\tau_r(d+1+r)$. As $d=7$ and $r\in \{0,2,4,6\}$, $d+1+r$ is even and the inequality $2r\leq d+1+r$ holds. 
Lemma~\ref{lemma:trianglefree} gives that $\tau_r(d+1+r)=0$, and Equation~\eqref{eq:w} yields $\mu(7,r)=\frac{r}{4}(64-r^2)$. Now, 
let us prove each assertion separately.
\begin{enumerate}[label=(\roman*)]
\item If $r=0$ then $\mu(d,r)=\mu(7,0)=0$. By Lemma~\ref{lemma:test}, $G=E_{8}^{(q)}$.
\item If $r=2$ then $\mu(d,r)=\mu(7,2)=30$. By Lemma~\ref{lemma:test}, $G=E_{8}^{(q-3)} \wedge S(26,18)$.
\item If $r=4$ then $\mu(d,r)=\mu(7,4)=48$. By Lemma~\ref{lemma:test}, $G=E_{8}^{(q-5)} \wedge S(44,36)$.
\item If $r=6$ then $\mu(d,r)=\mu(7,6)=42$. By Lemma~\ref{lemma:test}, $G=E_{8}^{(q-4)} \wedge S(38,30)$.
\end{enumerate}
Now, assume that Conjecture~\ref{conjecture:Saegusa} holds. 
As $S(26,18)$ is trace-minimal in $\mathcal{R}_{n+k}(2n)$ with $n=13$ and $k=5$ we would obtain that $S(26,18)=E_8 \wedge S(18,10)$. As $S(18,10)$ is trace-minimal in $\mathcal{R}_{n+1}(2n)$ with $n=5$, Proposition~\ref{proposition:saegusa} gives that $S(18,10)=E_8 \wedge C_{10}$. Therefore, if Conjecture~\ref{conjecture:Saegusa} holds then  $S(26,18)=E_8^{(2)} \wedge C_{10}$. 

As $S(44,36)$ is trace-minimal in $\mathcal{R}_{n+k}(2n)$ 
with $n=22$ and $k=14$, we would obtain that $S(44,36)=E_8 \wedge S(36,28)$. Similarly, $S(36,28)$ is a trace-minimal graph in $\mathcal{R}_{n+k}(2n)$ with $n=18$ and $k=10$ thus 
$S(36,28)=E_8 \wedge S(28,20)$. Again, as $S(28,20)$ is trace-minimal in $\mathcal{R}_{n+k}(2n)$ with $n=10$ and $k=2$, $S(28,20)=E_8 \wedge S(12,4)$. Therefore, if Conjecture~\ref{conjecture:Saegusa} holds then  $S(44,36) = E_8^{(4)} \wedge S(12,4)$.

Finally, $S(38,30)$ is trace-minimal in $\mathcal{R}_{n+k}(2n)$ 
with $n=19$ and $k=11$ thus $S(38,30)=E_8 \wedge S(30,22)$. Similarly, 
$S(30,22)$ is trace-minimal in $\mathcal{R}_{n+k}(2n)$ with $n=15$ and $k=7$ thus $S(30,22)=E_8 \wedge S(22,14)$. Again, $S(22,14)$ is trace-minimal  
in $\mathcal{R}_{n+k}(2n)$ with $n=11$ and $k=3$ thus $S(22,14)=E_8 \wedge S(14,6)$. Finally, by Lemma~\ref{lemma:bipartitem} the only trace-minimal graph in $\mathcal{R}_{n-1}(2n)$ is $K_{n,n}-nK_2$ thus $S(14,6)=K_{7,7}-7K_2$. Therefore, if Conjecture~\ref{conjecture:Saegusa} holds then $S(38,30) = E_8^{(3)} \wedge (K_{7,7}-7K_2)$. The lemma follows. 
\end{proof}

By Lemma~\ref{lemma:cases}, the graph $S(12,4)$ is already known, and it is depicted in Figure~\ref{fig:sporadic}. 
As a consequence, if Conjecture~\ref{conjecture:Saegusa} holds then Lemma~\ref{lemma:7} gives a full list of all trace-minimal graphs whose complement are $7$-regular. Except for few trace-minimal graphs to be determined, Lemma~\ref{lemma:6} gives trace-minimal graphs whose complement are $6$-regular. We close with an immediate consequence of Theorem~\ref{theorem:summary}.

\begin{corollary}
Each of the trace-minimal graphs $G$ defined in Lemma~\ref{lemma:6} or in Lemma~\ref{lemma:7} are $t$-optimal when the number of vertices is large enough 
(i.e., when the conditions of Theorem~\ref{theorem:summary} are satisfied).    
\end{corollary}

\section*{Acknowledgments}
This work is partially supported by City University of New York project entitled \emph{On the problem of characterizing graphs with maximum number of spanning trees}, grant number 66165-00. The first author wants to acknowledge Professors Martín D. Safe and Guillermo Durán for their valuable guidance and support.

\end{document}